\documentclass[11pt]{amsart}
\usepackage{amsmath,amssymb,amsfonts,latexsym,mathrsfs,amsthm, amscd, url}
\usepackage{verbatim}
\usepackage{graphicx, subfigure}
\usepackage{tikz}
\def\R{\mathbb{R}}
\def\Z{\mathbb{Z}}
\def\N{{\mathbb{N}}}

\def\C{{\mathcal{C}}}

\theoremstyle{definition}
\newtheorem{theorem}{Theorem}[section]
\newtheorem{proposition}[theorem]{Proposition}
\newtheorem{lemma}[theorem]{Lemma}
\newtheorem{corollary}[theorem]{Corollary}
\newtheorem{definition}[theorem]{Definition}
\newtheorem{conjecture}[theorem]{Conjecture}
\newtheorem{remark}[theorem]{Remark}

\newtheorem{example}[theorem]{Example}
\newtheorem*{acknowledgement}{Acknowledgements}
\newtheorem*{note}{Note}
\newtheorem*{thmtorussig}{Theorem \ref{Thm:L2torussig}}

\newcommand{\eps}{\varepsilon}
\begin{document}

\title{The $L^2$ signature of torus knots}
\author{Julia Collins}
\address{School of Mathematics\newline University of Edinburgh\newline James Clerk Maxwell Building\newline The King's Buildings\newline Mayfield Road\newline Edinburgh EH9 3JZ\newline Scotland, UK\newline \medskip}
\email{J.Collins-3@sms.ed.ac.uk}
\thanks{The author is supported by an EPSRC Doctoral Training Account.}

\begin{abstract}
We find a formula for the $L^2$ signature of a $(p,q)$ torus knot, which is the integral of the $\omega$-signatures over the unit circle.  We then apply this to a theorem of Cochran-Orr-Teichner to prove that the $n$-twisted doubles of the unknot, $n\neq 0,2$, are not slice.  This is a new proof of the result first proved by Casson and Gordon.
\end{abstract}

\maketitle

\begin{note}
It has been drawn to my attention that the main theorem of this paper, Theorem \ref{Thm:L2torussig}, was first proved in $1993$ by Robion Kirby and Paul Melvin \cite{KirbyMelvin94} using essentially the same method presented here.  The theorem has also recently been reproved using different techniques by Maciej Borodzik \cite{Borodzik09}. Despite the duplication of effort, I hope that readers will enjoy the exposition given here and the new corollaries which follow.
\end{note}

\section{Introduction}

Before we give any definitions of signatures or slice knots, let us first motivate the subject with a simple but difficult problem in number theory.  Suppose that you are given two coprime integers, $p$ and $q$, together with another (positive) integer $n$ which is neither a multiple of $p$ nor of $q$.  Write
\[ n= ap+bq, \quad a,b \in \Z, \quad 0 < a < q. \]
Now we ask the question:
\begin{center}
  ``Is $b$ positive or negative?''
\end{center}

Clearly, given any \emph{particular} $p$ and $q$, the answer is easy to work out, so the question is whether there is an (explicit) formula which could anticipate the answer.  Let us define
\[ j(n) = \begin{cases}\phantom{-}1 & \text{if $b>0$,}\\ -1 & \text{if $b<0$}\end{cases} \]
and let us study the sum
\[ s(n) = \sum_{i=1}^n j(i)\]
as $n$ varies between $1$ and $pq-1$.  It would not be an unreasonable first guess to suggest that $j(n)$ is $-1$ for the first $\lfloor \frac{pq}{2} \rfloor$ values of $n$, and $+1$ for the other half of the values.  Indeed, this is true when $p=2$ (see the nice `V' shape in Figure~\ref{fig:edge-a}).  But if we investigate other values of $p$ and $q$  then strange `wiggles' in the graph of $s$ start appearing (see Figures~\ref{fig:edge-b}, \ref{fig:edge-c} and \ref{fig:edge-d}).

\begin{figure}[htp]
  \begin{center}
    \subfigure[$p=2$, $q=19$]{\label{fig:edge-a}\includegraphics[width=8cm]{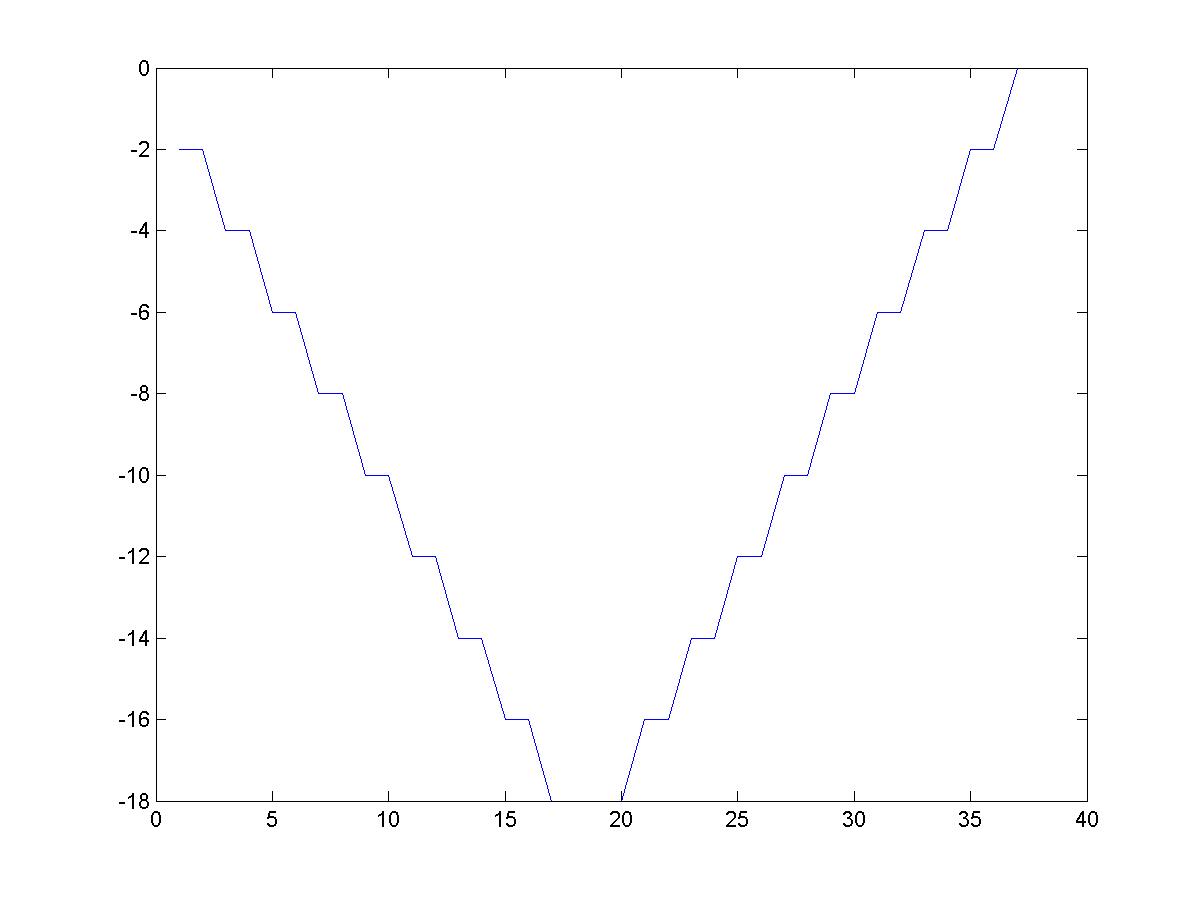}}
    \subfigure[$p=3$, $q=10$]{\label{fig:edge-b}\includegraphics[width=8cm]{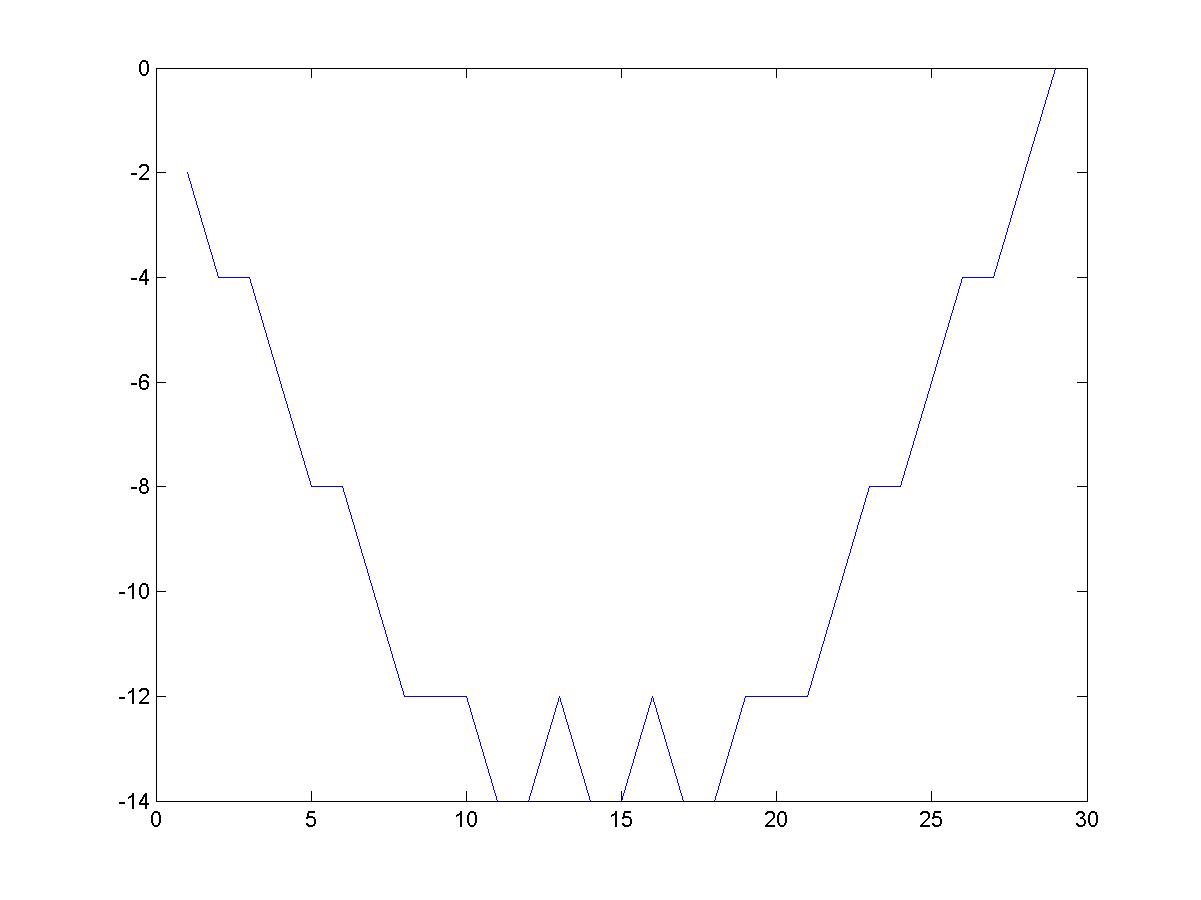}}
    \subfigure[$p=5$, $q=24$]{\label{fig:edge-c}\includegraphics[width=8cm]{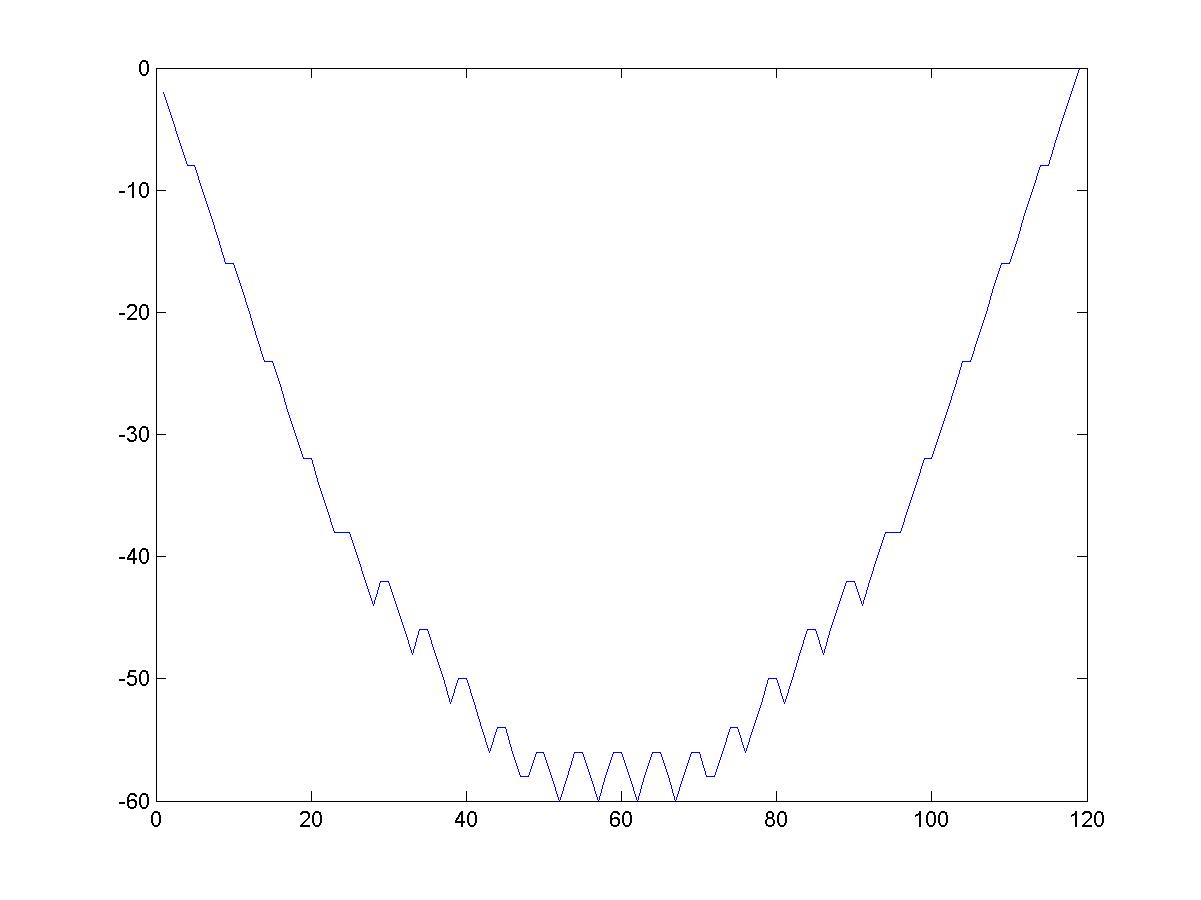}}
    \subfigure[$p=7$, $q=16$]{\label{fig:edge-d}\includegraphics[width=8cm]{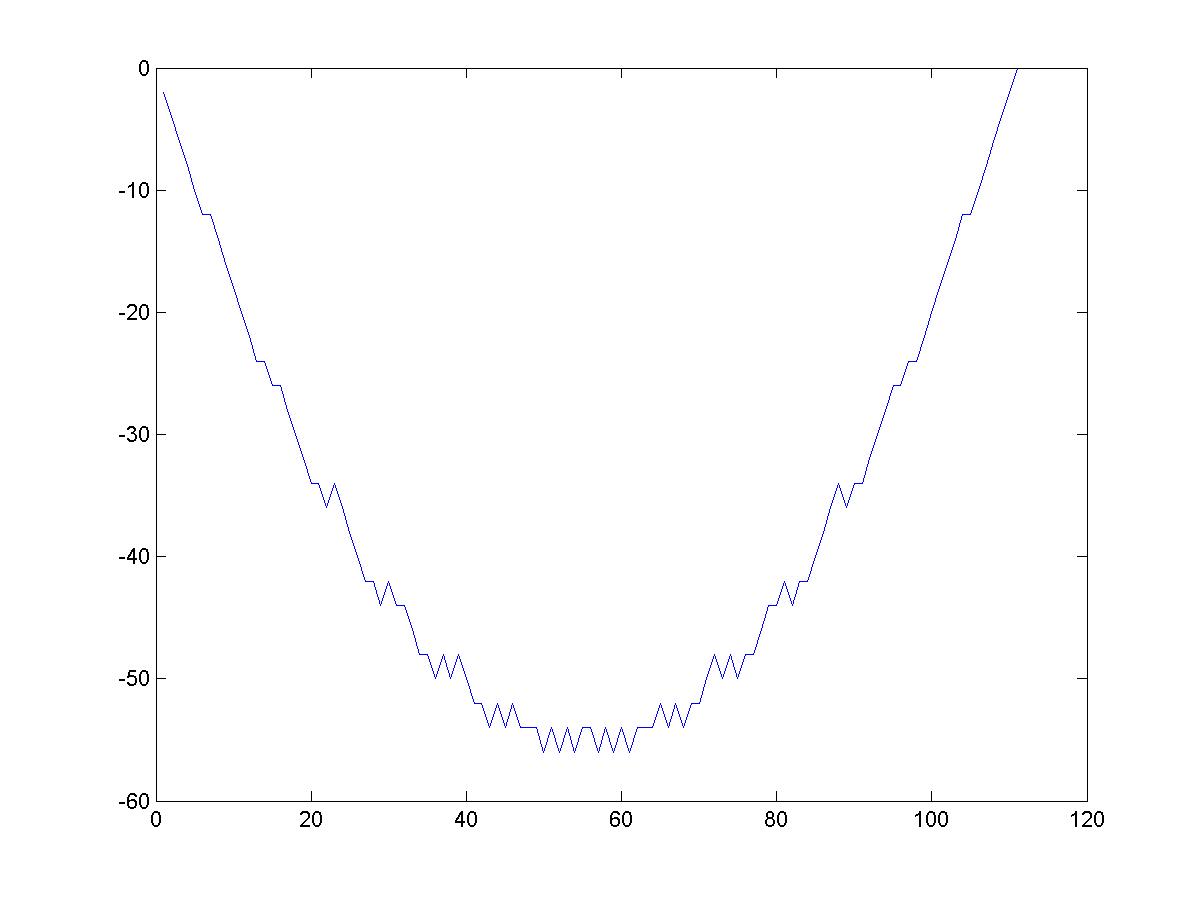}}
  \end{center}
  \caption{Graphs of $2s$ for various values of $p$ and $q$.}
  \label{graphs}
\end{figure}

\medskip

It turns out that the clue to finding the pattern is to realise that the function $j$ is the jump function of the $\omega$-signature of a $(p,q)$ torus knot.  What does this mean?  A torus knot $T_{p,q}$ is a knot which lives on the boundary of a torus, wrapping $p$ times around the meridian and $q$ times around the longitude. (If $p$ and $q$ are not coprime then $T_{p,q}$ is a link rather than a knot.)  Given a non-singular Seifert matrix $V$ for $T_{p,q}$ and a unit complex number $\omega$, the \emph{$\omega$-signature} $\sigma_\omega$ is the sum of the signs of the eigenvalues of the hermitian matrix
\[(1-\omega)V + (1-\overline{\omega})V^T.\]
This is independent of the choice of Seifert matrix.  The $\omega$-signature is an integer-valued function that is continuous (and therefore constant) everywhere except at the unit roots of the Alexander polynomial $\Delta_{p,q}(t) = \det(V-tV^T)$.  At these points, the signature `jumps', with the value of the jump at $\omega = e^{2\pi i n/pq}$ being given by $2j(n)$.

\medskip

The $\omega$-signature has proved to be useful in a variety of areas of mathematics; see Stoimenow's paper \cite{Stoimenow05} to get a comprehensive list, including applications to unknotting numbers of knots, Vassiliev invariants and algebraic functions on projective spaces. The jump function in particular appears to be related to the Jones polynomial~\cite{Garoufalidis03}.  But the most important use for signatures is in the study of knot concordance.

\medskip

A slice knot is one which is the boundary of a locally flat disc $D^2$ embedded into $D^4$.  We define the concordance group $\C$ to be the set of all knots $S^1 \hookrightarrow S^3$ under the equivalence relation $\sim$, where $K_1 \sim K_2$ if $K_2\#-K_1$ is slice.  Here $\#$ means connect sum and $-K_1$ is the mirror image of $K_1$ with the reverse orientation.  The structure of $\C$ remains a mystery, but recently Cochran, Orr and Teichner~\cite{COT03} having been probing its secrets using the techniques of $L^2$ signatures.  Amazingly, it turns out that a special case of these $L^2$ signatures turns out to be the integral of the $\omega$-signatures over the unit circle.

\medskip

Even more amazingly, it turns out that despite the $\omega$-signatures being fairly unpredictable for a torus knot, the \emph{integral} of the $\omega$-signatures has the following beautiful formula.

\begin{thmtorussig}
  Let $p$ and $q$ be coprime positive integers.  Then the integral of the $\omega$-signatures of the $(p,q)$ torus knot is
  \[ \displaystyle \int_{S^1} \sigma_{\omega} = -\frac{(p-1)(p+1)(q-1)(q+1)}{3pq} \text{ .} \]
\end{thmtorussig}

In this paper we prove Theorem \ref{Thm:L2torussig}, the reason for which is that the result can be combined with a theorem in \cite{COT03} to recover the old Casson-Gordon theorem that the twist knots are not slice.  The hope is that the techniques here may prove useful in investigating signatures of other families of knots and in proving more general theorems about the structure of the concordance group.

\medskip

\textbf{Structure of paper.} In Section 2 we give the definitions of the signatures and the associated jump functions we will be using.  Section 3 contains an analysis of the jump function of torus knots followed by the main result of the paper.  In Section 4 we apply this to the question of sliceness of twist knots and twisted doubles.

\section{Signatures and jump functions}

Let $K$ be a knot, $V$ be a Seifert matrix for $K$ of size $2g \times 2g$ and $\omega$ be a unit complex number.  The notation $\bar{\phantom{\omega}}$ denotes complex conjugation, whilst $^T$ means matrix transposition.

\medskip

We would like to define the $\omega$-signature to be the signature of $P:=(1-\omega)V + (1-\overline{\omega})V^T$.  However, notice that $P=(1-\omega)(V-\overline{\omega}V^T)$ and $\det P = (1-\omega)^{2g}\Delta_K(\overline{\omega})$ where $\Delta_K(t) := \det(V-tV^T)$ is the Alexander polynomial of $K$.  This means that $P$ becomes degenerate at the unit roots of the Alexander polynomial and we will need an alternative definition of the signature at these points.

\begin{definition}
For a unit complex number $\omega$ which is not a root of the Alexander polynomial $\Delta_K$, the \emph{$\omega$-signature} $\sigma_\omega(K)$ is the signature (i.e. the sum of the signs of the eigenvalues) of the hermitian matrix
  \[ P:=(1-\omega)V + (1-\overline{\omega})V^T \text{ .} \]
If $\omega$ is a unit root of $\Delta_K$, we define $\sigma_\omega(K)$ to be the average of the limit on either side.
\end{definition}

This concept was formulated independently by Levine~\cite{Levine69} and Tristram~\cite{Tristram69}; hence the $\omega$-signature is sometimes called the \emph{Levine-Tristram signature}.  It is a generalisation of the usual definition of a knot signature, i.e. the signature of $V+V^T$, which was developed by \cite{Trotter62, Murasugi65}.

\medskip

The function $\sigma_\omega$ is continuous as a function of $\omega$ except at roots of the Alexander polynomial .  Since $\sigma_\omega$ is integer-valued, this means that it is a step function with jumps at the roots of $\Delta_K$.

\begin{definition}
  The jump function $j_K \colon [0,1) \to \Z$ of a knot $K$ is defined by
    \[ j_K(x) = \displaystyle \frac{1}{2} \lim_{\eps \to 0} (\sigma_{\xi^+}(K) - \sigma_{\xi^-}(K)) \]
  where $\xi^+ = e^{2\pi i (x + \eps)}$ and $\xi^- = e^{2\pi i (x - \eps)}$ for $\eps>0$.
\end{definition}

\begin{lemma}
\label{Lemma:jump_properties}
  The jump function $j_K$ and the $\omega$-signature $\sigma_\omega(K)$ have the following properties.
  \begin{enumerate}
    \item $j_K(x) = 0$ if $e^{2\pi i x}$ is not a unit root of the Alexander polynomial of $K$.
    \item In particular, $j_K(0)=0$ and $\sigma_1(K) = 0$.
    \item $\sigma_\omega(K) = \sigma_{\overline{\omega}}(K)$ so $j_K(x) = -j_K(1-x)$.
    \item $\displaystyle \sigma_{e^{2\pi i x}}(K) = 2\sum_{y \in [0,x]} j_K(y)$ if $e^{2\pi i x}$ is not a root of the Alexander polynomial of $K$. (Notice that this is a finite sum because only finitely many of the jumps are non-zero.)
  \end{enumerate}
\end{lemma}

It has been known for some time that the usual knot signature $\sigma_{-1}$ vanishes for slice knots~\cite{Murasugi65} and that it is thus a concordance invariant.  The same is true for all the $\omega$-signatures (excepting the jump points).  In fact, it turns out that the integral of the $\omega$-signatures is a special case of a more powerful invariant.

\begin{definition}
An $L^2$-signature (or $\rho$-invariant) of a knot $K$  is a number $\rho(M,\phi) \in \R$ associated to a representation $\phi\colon \pi_1(M) \to \Gamma$, where $M$ is the zero-framed surgery on $S^3$ along $K$ and $\Gamma$ is a group.
\end{definition}

The precise definition is complicated and may be found in \cite[Section 5]{COT03}.  $L^2$ signatures are, in general, very difficult to compute.  However, if we pick a nice group for $\Gamma$ then magic happens and we get an explicit formula:

\begin{lemma}\cite[Lemma 5.4]{COT03}
  When $\Gamma = \Z$, we have that $\rho(M,\phi) = \int_{\omega \in S^1} \sigma_{\omega}(K)$, normalised to have total measure $1$.
\end{lemma}

Henceforth we shall refer to the integral of the $\omega$-signatures as \emph{the} $L^2$ signature of the knot.

\medskip

We end the section with a formula relating the $L^2$ signature of a knot to its jump function.

\begin{lemma}
\label{Lemma:L2asjump}
Suppose that the unit roots of the Alexander polynomial of a knot $K$ are $\omega_k = e^{2\pi i x_k}$ for $k=1, \dots, n$ and $x_1 < \dots < x_n$.  Then the $L^2$ signature of $K$ is
\[ \int_{\omega \in S^1} \sigma_{\omega}(K) = 2\sum_{i=1}^{n-1} \left(x_{i+1}-x_i \right) \sum_{k=1}^i j_K(x_k)\]
\end{lemma}

\begin{proof}
Let $\xi_k$ be any unit complex number between $\omega_k$ and $\omega_{k+1}$ for $k=1, \dots n-1$. Then we have that
\[ \int_{\omega \in S^1} \sigma_{\omega}(K) = \sum_{i=1}^{n-1} \left(x_{i+1}-x_i\right) \sigma_{\xi_i}\]
where we multiply by $(x_{i+1}-x_i)$ because that is the proportion of the unit circle which has signature $\sigma_{\xi_i}$. We now use (1) and (4) of Lemma \ref{Lemma:jump_properties} to rewrite $\sigma_{\xi_i}$ in terms of the jump function:
\[ \sigma_{\xi_i} = 2\sum_{y \in [0, x_i]} j_K(y)\, = 2\sum_{k=1}^{i} j_K(x_k)\]
\end{proof}

\begin{example}
To illustrate the notation in Lemma \ref{Lemma:L2asjump} we shall calculate the $L^2$ signature of the knot $K := 5_1$, otherwise known as the cinquefoil knot or the $(2,5)$ torus knot.  The Alexander polynomial of $K$ is
\[ 1-t+t^2-t^3+t^4 = \frac{(1-t^{10})(1-t)}{(1-t^2)(1-t^5)}\]
whose roots are the $10^{\text{th}}$ roots of unity that are neither $5^{\text{th}}$ roots of unity nor $-1$.  This means that the roots are $\omega_k =  e^{2\pi i x_k}$ where $x_1 = \frac{1}{10}$, $x_2 = \frac{3}{10}$, $x_3 = \frac{7}{10}$ and $x_4 = \frac{9}{10}$.

\medskip

Let $\xi_1 = e^{\frac{4}{5}\pi i}$, $\xi_2 = e^{\pi i}$ and $\xi_3 = e^{\frac{8}{5}\pi i}$.  Computing the $\omega$-signature at these points gives us $\sigma_{\xi_1} = -2$, $\sigma_{\xi_2} = -4$ and $\sigma_{\xi_3} = -2$.  We can thus draw the signature for every value on the unit circle:
\begin{figure}[h]
\begin{center}
\begin{tikzpicture}[scale=2.5,cap=round]
  \def\cosone{0.80901699}
  \def\sinone{0.58778525}
  \def\costhree{-0.30901699}
  \def\sinthree{0.951056516}
  \def\cosseven{-0.30901699}
  \def\sinseven{-0.951056516}
  \def\cosnine{0.80901699}
  \def\sinnine{-0.58778525}


   \node[above right] at (\cosone,\sinone) {$\omega_1$};
   \node[above left] at (\costhree,\sinthree) {$\omega_2$};
   \node[below left] at (\cosseven,\sinseven) {$\omega_3$};
   \node[below right] at (\cosnine,\sinnine) {$\omega_4$};

   \fill[blue!10!white] (0,0) -- (1,0) arc (0:36:1cm);
   \fill[blue!10!white] (0,0) -- (1,0) arc (0:-36:1cm);
   \fill[blue!20!white] (0,0) -- (\cosone,\sinone) arc (36:108:1cm);
   \fill[blue!20!white] (0,0) -- (\cosone,-\sinone) arc (-36:-108:1cm);
   \fill[blue!30!white] (0,0) -- (\costhree,\sinthree) arc (108:252:1cm);

   \draw (0,0) -- (\cosone,\sinone);
   \draw (0,0) -- (\costhree,\sinthree);
   \draw (0,0) -- (\cosseven,\sinseven);
   \draw (0,0) -- (\cosnine,\sinnine);

   \node at (0.5,0) {$\sigma_\omega = 0$};
   \node at (0.2, \sinone) {$\sigma_\omega = -2$};
   \node at (-0.5,0) {$\sigma_\omega = -4$};
   \node at (0.2, -\sinone) {$\sigma_\omega = -2$};

   \draw (0,0) circle (1cm);

   \draw[fill=black] (36:1cm) circle (0.01cm);
   \node[above] at (72:1cm) {$\xi_1$};
   \draw[fill=black] (72:1cm) circle (0.01cm);
   \draw[fill=black] (108:1cm) circle (0.01cm);
   \draw[fill=black] (-1,0) circle (0.01cm);
   \node[left] at (-1,0) {$\xi_2$};
   \draw[fill=black] (-36:1cm) circle (0.01cm);
   \node[below] at (-72:1cm) {$\xi_3$};
   \draw[fill=black] (-72:1cm) circle (0.01cm);
   \draw[fill=black] (-108:1cm) circle (0.01cm);
\end{tikzpicture}
\end{center}
\end{figure}

We can now compute the $L^2$ signature to be
\begin{eqnarray*}
\int_{\omega \in S^1} \sigma_\omega & = & (x_2 - x_1)\sigma_{\xi_1} + (x_3 - x_2)\sigma_{\xi_2} + (x_4 - x_3)\sigma_{\xi_3}\\
& = & \frac{2}{10}(-2) + \frac{4}{10}(-4) + \frac{2}{10}(-2)\\
& = & -\frac{12}{5}.
\end{eqnarray*}
\end{example}

\section{Torus knot signatures}

For coprime integers $p$ and $q$, the $(p,q)$ torus knot $T_{p,q}$ is the knot lying on the surface of a torus which winds $p$ times around the meridian and $q$ times around the longitude.  If $p$ and $q$ are not coprime, then $T_{p,q}$ is a link of more than one component.  The Alexander polynomial of $T_{p,q}$ is
  \[ \Delta_{p,q}(t) = \frac{(1-t^{pq})(1-t)}{(1-t^p)(1-t^q)}.\]
(A proof can be found in, for example, Lickorish~\cite[pg 119]{Lickorish}.)  The roots of this polynomial are the $pq^\text{th}$ roots of unity that are neither $p^\text{th}$ nor $q^\text{th}$ roots of unity.  This gives us $pq-p-q+1$ places at which the signature function could jump, namely $e^{2\pi i n/pq}$ for $n \in \Z$ with $0<n<pq$ such that $n$ is not divisible by $p$ or $q$.

\medskip

The jump functions of torus knots have been investigated by Litherland~\cite{Litherland79}.  His result is that
    \[ j_{p,q}\left(\frac{n}{pq}\right) = |L(n)| - |L(pq+n)| \]
  where $pq>n \in \N$ and
    \[ L(n) = \left\{ (i,j) \; | \; iq+jp = n, \,\, 0 \leq i\leq p, \, 0 \leq j \leq q \right\} \text{ .} \]

Notice that if $n$ is not a multiple of $p$ or $q$ then $L(n)$ and $L(pq+n)$ cannot both be nonempty.  To see this, suppose that $(i_1, j_1) \in L(n)$ and $(i_2,j_2) \in L(pq+n)$. Then $(i_2-i_1)q + (j_2-j_1)p = pq$, and since $p$ and $q$ are coprime we must have $i_2=i_1$ (mod $p$) and $j_2=j_1$ (mod $q$).  But this forces $i_1=i_2$ and $j_1=j_2$, which is a contradiction.  A similar argument shows that neither $L(n)$ nor $L(pq+n)$ can contain more than one element.  However, at least one of the two sets is nonempty. For, we can write $n=iq+jp$ with $0<i<p$, and if $j>0$ then $(i,j) \in L(n)$ whilst if $j<0$ we have $(i,j+q) \in L(pq+n)$.

\medskip

If $n$ is a multiple of $p$ or $q$ then $|L(n)| = 1 = |L(pq+n)|$.  Putting these results together gives us the following.

\begin{proposition}
\label{Prop:torusjump}
  The jump function of the $(p,q)$ torus knot is
  \[ j_{p,q}\left(\frac{n}{pq}\right) =
    \begin{cases}
      +1 & \text{ if } |L(n)|=1 \\
      -1 & \text{ if } |L(n)|=0 \\
      0  & \text{ if } n \text{ is a multiple of } p \text{ or } q
    \end{cases} \]
\end{proposition}

We need one more lemma before we are ready to find a formula for the $L^2$ signature.

\begin{lemma}
\label{Lemma:oneof}
  If $p$ and $q$ are coprime and $1 \leq n \leq pq-1$ with $n$ not a multiple of $p$ or $q$, then exactly one of $n$ and $pq-n$ can be written as $iq+jp$ for $i,j >0$.
\end{lemma}

\begin{proof}
  See, for example, \cite[Lemma 1.6]{BeckSinai}.
\end{proof}

\begin{theorem}
\label{Thm:L2torussig}
  Let $p$ and $q$ be coprime positive integers.  Then the $L^2$ signature of the $(p,q)$ torus knot is
  \[ \displaystyle \int_{S^1} \sigma_{\omega} = -\frac{(p-1)(p+1)(q-1)(q+1)}{3pq} \text{ .} \]
\end{theorem}

\begin{proof}
  Denote the jump function of the $(p,q)$ torus knot by $j_{p,q}$. The signature function at $\omega$ can be defined as the sum of the jump functions up to that point (Lemma \ref{Lemma:jump_properties}).  If $\omega_n := e^{2\pi i x}$ with $x \in (\frac{n}{pq},\frac{n+1}{pq})$ then
  \[ \sigma_{\omega_n}(T_{p,q}) = \displaystyle 2 \sum_{i=1}^n j_{p,q}\left(\frac{i}{pq}\right) \]

  We can now use Lemma \ref{Lemma:L2asjump} to find a formula for the $L^2$ signature in terms of the jump function.

  \begin{eqnarray}
  \int_{S^1} \sigma_\omega & = & \sum_{n=1}^{pq-1} \frac{1}{pq} (\sigma_{\omega_n})\\
  & = & \frac{2}{pq} \sum_{n=1}^{pq-1}\sum_{i=1}^n j_{p,q}\left(\frac{i}{pq}\right) \\
  & = & \frac{2}{pq} \left( j_{p,q}\left(\frac{1}{pq}\right) + \left(j_{p,q}\left(\frac{1}{pq}\right)+j_{p,q}\left(\frac{2}{pq}\right)\right) + \dots + \sum_{i=1}^{pq-1}j_{p,q}\left(\frac{i}{pq}\right)\right)\\
  \label{eqn:sigint} & = & \frac{2}{pq} \sum_{n=1}^{pq-1} (pq-n)\,j_{p,q}\left(\frac{n}{pq}\right)
  \end{eqnarray}

  Let $S$ be the set defined by
  \[ \bigg\{  n \in \{1,\dots, pq-1\}\; | \;  n=qx+py, \,\, 0 < x < p,\,\, 0 < y < q \bigg\} \]

  Given an integer $n \in \{1,\dots, pq-1\}$ which is not a multiple of $p$ or $q$, we can write $n=qx+py$ with $0<x<p$. By Lemma \ref{Lemma:oneof}, either $n \in S$ or $pq-n \in S$.  Proposition \ref{Prop:torusjump} tells us that in the first case we have $j_{p,q}(n/pq) = 1$, whilst in the second case we have $j_{p,q}(n/pq) = -1$.  If $n$ is a multiple of $p$ or $q$ then the jump function will be zero and so it will not contribute to the sum.

  \medskip

  We may rewrite equation (\ref{eqn:sigint}) as

  \[ \int_{S^1} \sigma_\omega = \frac{2}{pq}\left(\sum_{n \in S}(pq-n) - \sum_{n \in S}n\right) = \frac{2}{pq}\sum_{n \in S}(pq-2n) \text{ .} \]

  There are $\frac{1}{2}(p-1)(q-1)$ points in $S$, and in the paper by Mordell \cite{Mordell51} we find the following formula
  \[ \sum_{n \in S} n = \frac{1}{3}pq(p-1)(q-1) + \frac{1}{12}(p-1)(q-1)(p+q+1) \text{ .} \]

  Putting these together gives us

  \begin{eqnarray*}
    \int_{S^1} \sigma_\omega & = & \frac{2}{pq}\left(\sum_{n \in S}pq - 2\sum_{n \in S} n\right)\\
    & = & \frac{2}{pq}\left(\frac{1}{2}(p-1)(q-1)pq - 2\left(\frac{1}{3}pq(p-1)(q-1) + \frac{1}{12}(p-1)(q-1)(p+q+1)\right) \right)\\
    & = & \frac{1}{pq}(p-1)(q-1) \left(pq - \frac{4}{3}pq - \frac{1}{3}(p+q+1)\right) \\
    & = & -\frac{1}{3pq}(p-1)(q-1)(pq+p+q+1) \\
    & = & -\frac{1}{3pq}(p-1)(q-1)(p+1)(q+1)
  \end{eqnarray*}

\end{proof}

\begin{remark}
  That there is such a neat formula for the $L^2$ signature of a torus knot is all the more surprising considering the absence of an explicit formula for the usual signature $\sigma_{-1}$ of a torus knot.  There is only the following formula due to Hirzebruch~\cite{Hirzebruch67} for $p$ and $q$ odd and coprime:
    \[ \sigma_{-1}(T_{p,q}) = -\left(\frac{(p-1)(q-1)}{2} + 2(N_{p,q} + N_{q,p})\right)\]
    where
    \[ N_{p,q} = \#\left\{ (x,y) \:| \: 1 \leq x \leq \frac{p-1}{2}, \,\, 1 \leq y \leq \frac{q-1}{2}, \,\, -\frac{p}{2} < qx-py <0. \right\}\]
    Further work was done by Brieskorn~\cite{Brieskorn66} and Gordon/Litherland/Murasugi~\cite{GordonLitherlandMurasugi81}, but there appears to be no nicer formula for the signature of a torus knot.
\end{remark}

\section{Twist knots}

As an important corollary, we show that the twist knots $K_n$ are not slice.  This was proved in the 1970s by Casson and Gordon~\cite{CassonGordon86} but the following proof, which uses a result of Cochran, Orr and Teichner, is much shorter and simpler.\footnote{It is also an interesting historical point that Milnor used an early version of the $\omega$-signatures to show that an infinite number of the $K_n$ are independent in the concordance group $\C$~\cite{Milnor68}.}

\begin{definition}
The \emph{twist knots} $K_n$ are the following family of knots:
  \begin{center}
    \includegraphics[width=6cm]{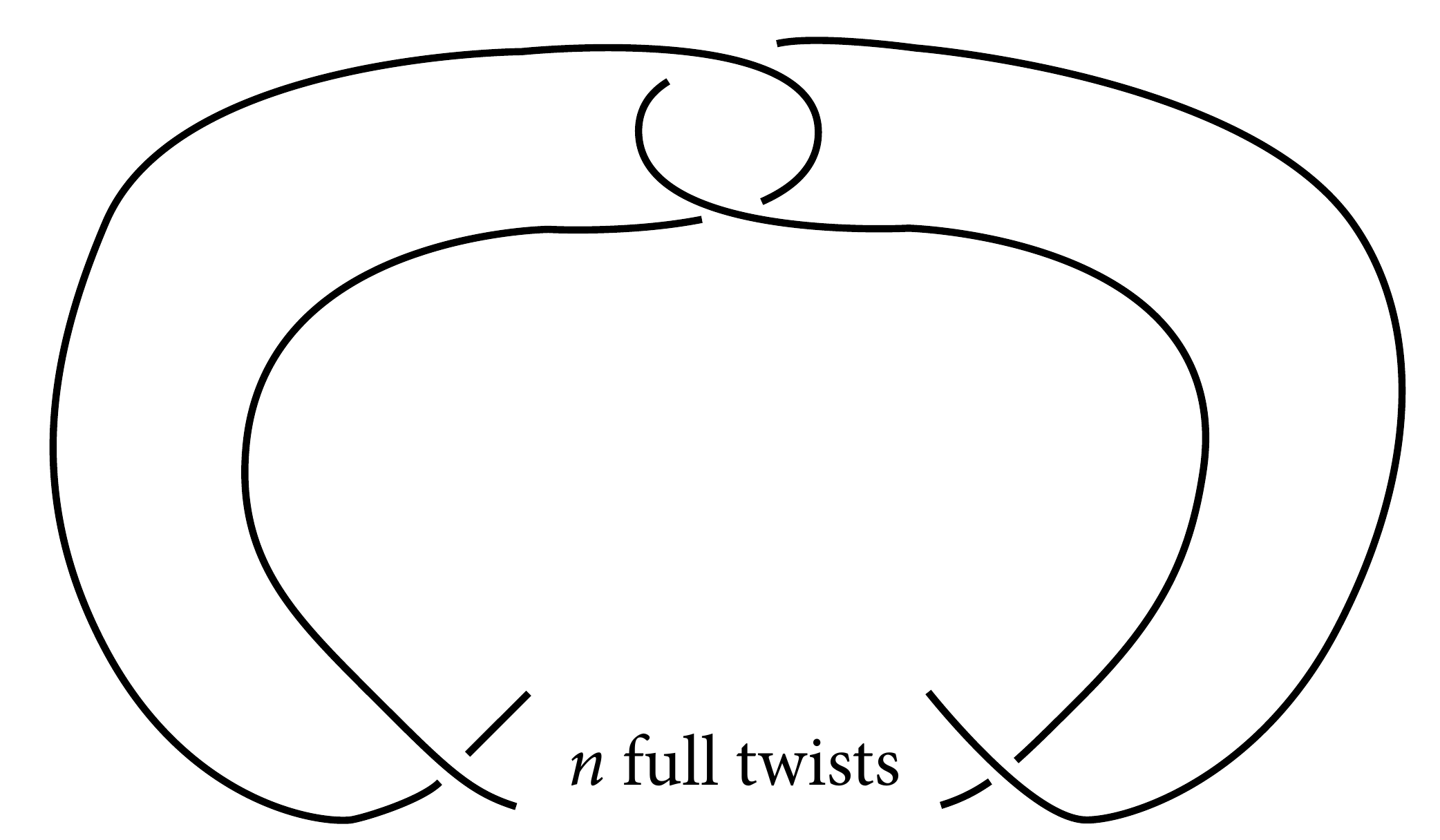}
  \end{center}
For example, $K_{-1}$ is the trefoil, $K_1$ is the figure-eight knot and $K_2$ is Stevedore's knot $6_1$.  The knot $K_n$ is sometimes called the \emph{$n$-twisted double of the unknot}.
\end{definition}

A Seifert matrix for $K_n$ is
  \[V = \left(\begin{array}{cc} -1 & 1\\ 0 & n\end{array}\right)\]
which gives the Alexander polynomial as $-nt^2+(2n+1)t-n$.  There is a class of twist knots (those for which $n=m(m+1)$ for some $m$) which are algebraically slice.  This means that there is simple closed curve $\gamma$ on the Seifert surface $F$ such that $\gamma$ is nontrivial in $H_1(F)$ and such that $\gamma^+$, which is the curve pushed off the Seifert surface, has zero linking with $\gamma$.  The consequence of this is that all signatures and other known slice invariants vanish.  The question is then: are these knots \emph{really} slice?

\medskip

The following theorem shows us that one way of finding the answer is to consider the slice properties of the curve $\gamma$ rather than those of the original knot.

\begin{theorem}[\cite{COT03}]
\label{Thm:genusonesurface}
  Suppose $K$ is a $(1.5)$-solvable knot with a genus one Seifert surface $F$. Suppose that the classical Alexander polynomial of $K$ is non-trivial. Then there exists a homologically essential simple closed curve $J$ on $F$, with self-linking zero, such that the integral over the circle of the $\omega$-signature function of $J$ (viewed as a knot) vanishes.
\end{theorem}

\begin{corollary}
\label{Cor:twistslice}
  The twist knots $K_n$ are not slice unless $n=0$ or $n=2$.
\end{corollary}

\begin{proof}
  The Alexander polynomial of $K_n$ is $-nt^2 - (2n+1)t +n$.  If $n<0$ then $\Delta_{K_n}$ has distinct roots on the unit circle and an easy computation shows that the signature is non-zero.  If $n>0$ then $\Delta_{K_n}$ is reducible if and only if $4n+1$ is a square.  Since the Alexander polynomial of a slice knot has the form $f(t)f(t^{-1})$ \cite{FoxMilnor66}, it follows that $K_n$ cannot be slice if $4n+1$ is not a square.

  \medskip

  Suppose $4n+1$ = $l^2$ with $l=2m+1$.  Then $n=m(m+1)$.  Using the obvious genus $1$ Seifert surface $F$ for $K_{m(m+1)}$ it can be seen that the only simple closed curve on $F$ with self-linking zero is the $(m,m+1)$ torus knot (see, for example, Kauffman~\cite[Chapter VIII]{Kauffman}).  Since the $L^2$ signature for any torus knot is non-zero (except for $m=0,-1,1,-2$) by Theorem \ref{Thm:L2torussig}, this means that $K_{m(m+1)}$ cannot be (1.5)-solvable and therefore not slice unless $n=0$ or $n=2$.
\end{proof}

\begin{corollary}
\label{Cor:twisteddouble}
Let $K$ be a knot and $D_n(K)$ the $n$-twisted double ($n \neq 0$) of $K$ as shown in Figure 2.
\begin{itemize}
\item[(a)] $D_n(K)$ cannot be slice unless $n=m(m+1)$ for some $m \in \Z$ and $\int_{S^1} \sigma_\omega(K) = \frac{(m-1)(m+2)}{3}$.  In particular, $D_2(K)$ can only be slice if $\int_{S^1} \sigma_\omega(K)=0$.
\item[(b)] For any given $K$ with $\int_{S^1} \sigma_\omega(K) \neq 0$, there is at most one $D_n(K)$ which can be slice.
\end{itemize}
\end{corollary}

\begin{figure}
  \label{fig:twisteddouble}
  \centering
  \includegraphics[width=8cm]{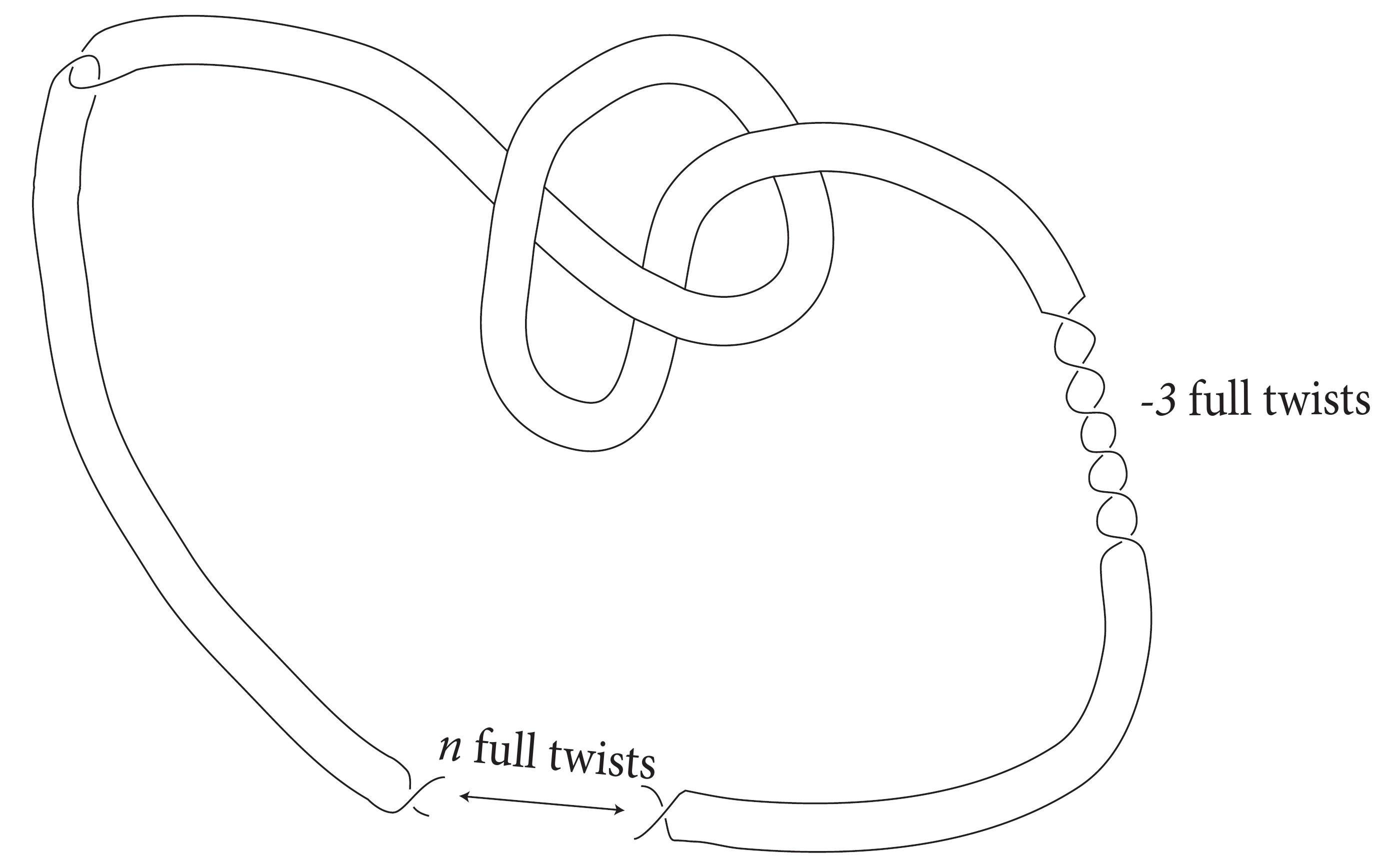}
  \caption{The $n$-twisted double of the right-handed trefoil.}
\end{figure}

\begin{proof}
The Alexander polynomial of $D_n(K)$ is once again $-nt^2 - (2n+1)t +n$ and the same argument as in the proof of Corollary \ref{Cor:twistslice} shows that $D_n(K)$ is algebraically slice if and only if $n=m(m+1)$ for some integer $m$. (Notice that if $n=0$ then the Alexander polynomial is trivial and $D_0(K)$ is slice by Freedman's work \cite{FreedmanQuinn}.)  The zero-framed curve on the obvious Seifert surface is the connected sum of $K$ and the $(m,m+1)$ torus knot, $K \# T_{(m,m+1)}$.  If we denote the $L^2$ signature by $s$, we have
  \begin{eqnarray*}
    s(K \# T_{m(m+1)}) & = & s(K) + s(T_{(m,m+1)})\\
                       & = & s(K) - \frac{(m-1)(m+1)m(m+2)}{3m(m+1)}\\
                       & = & s(K) - \frac{(m-1)(m+2)}{3}
  \end{eqnarray*}
By Theorem \ref{Thm:L2torussig}, $D_n(K)$ can only be slice if $s(K) = \frac{(m-1)(m+2)}{3}$.  In particular, if $m=1$ or $m=-2$ then $T_{(m,m+1)}$ is the unknot and so $s(K)$ must be zero for $D_2(K)$ to be slice.  This proves (a).
For (b), suppose that $3s(K) = (m-1)(m+2) \neq 0$.  Rearranging, we get $m^2+m - 2-3s(K)=0$.  Suppose that $m_1$ and $m_2$ are roots.  Then $m_1+m_2 = -1$, so $m_1(m_1+1) = -(m_2+1)(-m_2) = m_2(m_2+1)$, giving only one value for $n$.
\end{proof}

In the paper \cite{Kim05} Kim proves that for any knot $K$, all but finitely many algebraically slice twisted doubles of $K$ are linearly independent in the concordance group $\C$.  Using our theorem we can conjecture that there is a much stronger result about the independence of the twisted doubles of $K$.

\begin{conjecture} For a fixed knot $K$, the $D_{m(m+1)}(K)$ are linearly independent in $\C$ for all but one (or two, if $\int_{S^1} \sigma_\omega(K) = 0$) values of $m(m+1)$.
\end{conjecture}

The proof of this conjecture will require Theorem \ref{Thm:genusonesurface} to be extended to connected sums of genus one Seifert surfaces.

\begin{acknowledgement}
I would like to thank Andrew Ranicki for suggesting this as an interesting problem and for supplying me with the correct ingredients to solve it.  I would also like to thank Berian James for helping me to first find the formula by unconventional means, Peter Teichner for suggesting the application to slicing the twist knots, and Matthew Heddon for spotting a mistake in Corollary \ref{Cor:twisteddouble}.
\end{acknowledgement}

\bibliographystyle{amsalpha} \bibliography{bibfile}

\providecommand{\bysame}{\leavevmode\hbox to3em{\hrulefill}\thinspace}
\providecommand{\MR}{\relax\ifhmode\unskip\space\fi MR }
\providecommand{\MRhref}[2]{%
  \href{http://www.ams.org/mathscinet-getitem?mr=#1}{#2}
}
\providecommand{\href}[2]{#2}
\begin{thebibliography}{GLM81}

\bibitem[Bor09]{Borodzik09}
Maciej Borodzik, \emph{A rho-invariant of iterated torus knots},
  arXiv:math.AT/0906.3660 (2009).

\bibitem[BR07]{BeckSinai}
Matthias Beck and Sinai Robins, \emph{Computing the continuous discretely},
  Undergraduate Texts in Mathematics, Springer, New York, 2007.

\bibitem[Bri66]{Brieskorn66}
Egbert Brieskorn, \emph{Beispiele zur {D}ifferentialtopologie von
  {S}ingularit\"aten}, Invent. Math. \textbf{2} (1966), 1--14.

\bibitem[CG86]{CassonGordon86}
A.~J. Casson and C.~McA. Gordon, \emph{Cobordism of classical knots}, \`A la
  recherche de la topologie perdue, Progr. Math., vol.~62, Birkh\"auser Boston,
  1986, With an appendix by P. M. Gilmer, pp.~181--199.

\bibitem[COT03]{COT03}
Tim~D. Cochran, Kent~E. Orr, and Peter Teichner, \emph{Knot concordance,
  {W}hitney towers and {$L\sp 2$}-signatures}, Ann. of Math. (2) \textbf{157}
  (2003), no.~2, 433--519.

\bibitem[FM66]{FoxMilnor66}
R.~H. Fox and John~W. Milnor, \emph{Singularities of {$2$}-spheres in
  {$4$}-space and cobordism of knots}, Osaka J. Math. \textbf{3} (1966),
  257--267.

\bibitem[FQ90]{FreedmanQuinn}
Michael~H. Freedman and Frank Quinn, \emph{Topology of 4-manifolds}, Princeton
  Mathematical Series, vol.~39, Princeton University Press, Princeton, NJ,
  1990.

\bibitem[Gar03]{Garoufalidis03}
Stavros Garoufalidis, \emph{Does the {J}ones polynomial determine the signature
  of a knot?}, arXiv:math/0310203 (2003).

\bibitem[GLM81]{GordonLitherlandMurasugi81}
C.~McA. Gordon, R.~A. Litherland, and Kunio Murasugi, \emph{Signatures of
  covering links}, Canad. J. Math. \textbf{33} (1981), no.~2, 381--394.

\bibitem[Hir95]{Hirzebruch67}
Friedrich Hirzebruch, \emph{Singularities and exotic spheres}, S\'eminaire
  {B}ourbaki 1966/67, {V}ol.\ 10, Soc. Math. France, Paris, 1995, pp.~Exp.\
  No.\ 314, 13--32.

\bibitem[Kau87]{Kauffman}
Louis Kauffman, \emph{On knots}, Princeton University Press, 1987.

\bibitem[Kim05]{Kim05}
Se-Goo Kim, \emph{Polynomial splittings of {C}asson-{G}ordon invariants}, Math.
  Proc. Cambridge Philos. Soc. \textbf{138} (2005), no.~1, 59--78.

\bibitem[KM94]{KirbyMelvin94}
Robion~C. Kirby and Paul Melvin, \emph{Dedekind sums, {$\mu$}-invariants and
  the signature cocycle}, Math. Ann. \textbf{299} (1994), no.~2, 231--267.

\bibitem[Lev69]{Levine69}
J.~Levine, \emph{Knot cobordism groups in codimension two}, Comment. Math.
  Helv. \textbf{44} (1969), 229--244.

\bibitem[Lic97]{Lickorish}
W.~B.~Raymond Lickorish, \emph{An introduction to knot theory}, Springer
  Graduate Texts in Mathematics, 1997.

\bibitem[Lit79]{Litherland79}
R.~A. Litherland, \emph{Signatures of iterated torus knots}, Topology of
  low-dimensional manifolds ({P}roc. {S}econd {S}ussex {C}onf., {C}helwood
  {G}ate, 1977), Lecture Notes in Math., vol. 722, Springer, 1979, pp.~71--84.

\bibitem[Mil68]{Milnor68}
John~W. Milnor, \emph{Infinite cyclic coverings}, Conference on the Topology of
  Manifolds (Michigan State Univ., E. Lansing, Mich., 1967), Prindle, Weber \&
  Schmidt, Boston, Mass., 1968, pp.~115--133.

\bibitem[Mor51]{Mordell51}
L.~J. Mordell, \emph{The reciprocity formula for {D}edekind sums}, Amer. J.
  Math. \textbf{73} (1951), 593--598.

\bibitem[Mur65]{Murasugi65}
Kunio Murasugi, \emph{On a certain numerical invariant of link types}, Trans.
  Amer. Math. Soc. \textbf{117} (1965), 387--422.

\bibitem[Sto05]{Stoimenow05}
A.~Stoimenow, \emph{Some applications of {T}ristram-{L}evine signatures and
  relation to {V}assiliev invariants}, Adv. Math. \textbf{194} (2005), no.~2,
  463--484.

\bibitem[Tri69]{Tristram69}
A.~G. Tristram, \emph{Some cobordism invariants for links}, Proc. Cambridge
  Philos. Soc. \textbf{66} (1969), 251--264.

\bibitem[Tro62]{Trotter62}
H.~F. Trotter, \emph{Homology of group systems with applications to knot
  theory}, Ann. of Math. (2) \textbf{76} (1962), 464--498.

\end{thebibliography}

\end{document}